\documentclass[12pt]{article}
\usepackage{amsmath,amssymb,amsthm}
\usepackage{url}
\usepackage{geometry}
\usepackage{indentfirst}

\title{Elementary Proof of the Completeness of OEIS A051221 Below 2000}
\author{Seiichi Azuma\\
\small\texttt{seazuma1@gmail.com}
}
\date{}

\theoremstyle{plain}
\newtheorem{theorem}{Theorem}[section]
\newtheorem{lemma}[theorem]{Lemma}

\theoremstyle{definition}

\theoremstyle{remark}

\numberwithin{equation}{section}

\begin{document}

\maketitle

\begin{abstract}
The OEIS sequence A051221 consists of nonnegative integers of the form $10^x - y^2$. The known values are those less than or equal to 2000 with $x \leq 7$, and it is conjectured that no new values in this range appear for $x \geq 8$. In this paper, we give an elementary proof that this is indeed the case by using Pell-type equations.
\end{abstract}

\section{Introduction}

The sequence A051221 in the OEIS\footnote{\url{https://oeis.org/A051221}} is defined as the set of all nonnegative integers that can be expressed in the form $10^x - y^2$ for integers $x \geq 0$ and $y \geq 0$. Currently, the sequence consists of
\[
0, 1, 6, 9, 10, 19, 36, 39, \dots, 1999.
\]
The current list was obtained by computing the values in the set
\[
T(X) = \bigcup_{x=1}^{X} \left\{10^x - y^2 \mid y \in \mathbb{Z} \right\} \cap [0,2000],
\]
and stopping when $T(X+1) = T(X)$. Since this equality holds at $X = 7$, the currently known elements of the sequence correspond to values of $x \leq 7$. However, this does not immediately imply that new values cannot appear for $x \geq 9$. Intuitively, it is unlikely that a square is very close to a power of 10, but proving this rigorously is nontrivial. In this paper, we present an elementary proof that $T(X) = T(7)$ for all $X \geq 7$, thereby confirming that the known list is complete within the range $[0,2000]$.

Generally, equation $a^x - y^2 = c$ is known as a generalized Lebesgue-Ramanujan-Nagell equation. A comprehensive survey by Maohua Le and Gökhan Soydan~\cite{LeSoydan2020} summarizes the known results, including the finiteness of integer solutions $(x, y)$ for fixed $a$ and $c$. For specific cases, Christian Hercher and Karl Fegert provided an elementary proof using Pell-type equations for expressions of this form~\cite{HercherFegert2025}. Our method follows a similar method.

\section{Reformulation and Strategy}
Note that when $x$ is even, the smallest possible positive value of the form $10^x-y^2$ is
\[
10^x - (10^{x/2} - 1)^2 = 2 \cdot 10^{x/2} - 1,
\]
which exceeds 2000 whenever $x/2 \geq 4$, that is, $x \geq 8$. Therefore, it is sufficient to show that no additional values of the form $10^x - y^2$ can appear for odd $x \geq 9$.

We fix an integer $c \in [0,2000]$ that is not known to be in the sequence. 
Let $x = 2u + 1$. Then the equation becomes
\[
10 \cdot (10^u)^2 - y^2 = c.
\]

This is known as Pell-type equation and can be analyzed via quadratic fields. As shown in the next section, any solution $10t^2 - s^2 = c$ can be expressed with integers $a,b,K$ satisfying:
\[
s + t \sqrt{10} = (a + b \sqrt{10})(19 + 6 \sqrt{10})^K, \quad |a| \leq 3 \sqrt{c}, \quad \sqrt{c/10} \leq |b| \leq \sqrt{c}, \quad 10b^2 - a^2 = c.
\]
For each $c$, the pair $(a,b)$ that satisfies the condition is finite, so our approach is to show that $t$ cannot be the form of $10^u$ for any integer $k$ for each $(a,b)$. We consider the sequence of integers $t_k$ defined by
\[
(a + b \sqrt{10})(19 + 6 \sqrt{10})^k = s_k + t_k \sqrt{10},
\]
and show the impossibility of $t_k = 10^u$ for any $k$ and $u \geq 4$ by showing that $t_k \equiv0 \pmod {10^4}$ and $t_k \equiv 10^m \pmod p$ do not hold simultaneously for any $m$ for some $p$.
Note that negating $a,b$ results in negating $s_k,t_k$ and that negating $b,k$ corresponds to taking the conjugate, which negates $t_k$. Hence, it is sufficient to consider only nonnegative $a$ and $b$, and to prove the impossibility of $t_k = \pm 10^u$.

\section{Main Lemma}

The key step in our argument is the following lemma.

\begin{lemma}
Let $c$ be a fixed positive integer. Suppose that there exist integers $s,t$ that satisfy the equation
\[
10t^2 - s^2 = c.
\]
Then there exist integers $a, b, K$ that satisfy
\[
s + t \sqrt{10} = (a + b \sqrt{10})(19 + 6 \sqrt{10})^K, \quad 0 \leq |a| \leq 3 \sqrt{c}, \quad \sqrt{\frac{c}{10}} \leq |b| \leq \sqrt{c}, \quad 10b^2 - a^2 = c.
\]
\end{lemma}

\begin{proof}
We first assume $s + t \sqrt{10} > 0$. We define $K$ as
\[
K = \left\lfloor \frac{\log(s + t \sqrt{10}) - \log \sqrt{c}}{\log(19 + 6 \sqrt{10})} + \frac{1}{2} \right\rfloor.
\]
This implies
\[
(19 + 6 \sqrt{10})^{k - 1/2} \leq \frac{s + t \sqrt{10}}{\sqrt{c}} \leq (19 + 6 \sqrt{10})^{k + 1/2}.
\]
We then define $a,b$ by
\[
a + b \sqrt{10} = \frac{s + t \sqrt{10}}{(19 + 6 \sqrt{10})^K}.
\]
Since $19 + 6 \sqrt{10}$ is a unit in $\mathbb{Z}[\sqrt{10}]$, $a$ and $b$ are integers, and we have
\[
10b^2 - a^2 = 10t^2 - s^2 = c.
\]

Note that $\sqrt{19 + 6 \sqrt{10}} = 3 + \sqrt{10}$, so the inequality above implies
\[
\sqrt{c}(-3 + \sqrt{10}) \leq a + b \sqrt{10} \leq \sqrt{c}(3 + \sqrt{10}).
\]
Now consider the set of real numbers of the form $a + b \sqrt{10}$ satisfying this inequality. These values correspond to the points $(a, b)$ in the hyperbola
\[
10b^2 - a^2 = c
\]
within a bounded region. By analyzing the endpoints, we have
\[
-3\sqrt{c} \leq a \leq 3\sqrt{c}, \quad \sqrt{c/10} \leq b \leq \sqrt{c},
\]
since the smallest value of $b$ occurs when $a = 0$, which gives $10b^2 = c$.

When $s + t \sqrt{10} < 0$, we apply the same discussion to $-s - t \sqrt{10}$ and conclude similarly that
\[
-3\sqrt{c} \leq a \leq 3\sqrt{c}, \quad \sqrt{c/10} \leq -b \leq \sqrt{c}.
\]
\end{proof}

\section{Example}

Let $c = 31$, which is not in the known list of values. Then, the only integer pair $(a,b) = (3,2)$ satisfies the conditions
\[
0 \leq a \leq 3 \sqrt{c}, \quad \sqrt{\frac{c}{10}} \leq b \leq \sqrt{c}, \quad 10b^2 - a^2 = c.
\]

Define the sequence of integers $t_k$ by
\[
(a + b \sqrt{10})(19 + 6 \sqrt{10})^k = s_k + t_k \sqrt{10}.
\]

The sequence $\{t_k\}$ satisfies the recurrence:
\[
t_0 = b = 2, \quad t_1 = 6a + 19b = 56, \quad t_{k+2} = 38 t_{k+1} - t_k.
\]

Let $N = 10^4$ and $p = 160001$.  The sequences $(t_k \bmod N)$ and $(t_k \bmod p)$ are linear recurrence sequences modulo fixed moduli and, therefore, eventually periodic. In this case, these sequences have a common period $40000$, and we find that $t_k \equiv 0 \pmod{N}$ occurs precisely when $k \equiv 3309 \pmod{5000}$. By computing $t_k \bmod p$ at those values of $k$, we have
\[
t_k \equiv 4354,\, 121626,\, 16949,\, 146265,\, 155647,\, 38439,\, 143052,\, 13736 \pmod{p}.
\]

We also observe that the set of residues $\{\pm 10^m \bmod p\}$ forms a multiplicative subgroup of order $1250$ in $\mathbb{F}_p^\times$, and none of the eight residues above belongs to this subgroup.

Therefore, we conclude that $t_k = \pm 10^u$ does not occur for any $k$ and $u \geq 4$, which confirms that $c = 31$ cannot appear in the form $10^x - y^2$ for any $x \geq 9$.

\section{Whole Computation}

We now describe the complete computation that verifies the completeness of the sequence A051221 within the interval $[0,2000]$.

For every integer $c \in [0,2000]$ that is not contained in the known list of the form $10^x - y^2$ for $x \leq 7$, we enumerate all integer pairs $(a,b)$ satisfying the following conditions:
\[
0 \leq a \leq 3 \sqrt{c}, \quad \sqrt{\frac{c}{10}} \leq b \leq \sqrt{c}, \quad 10b^2 - a^2 = c.
\]
For each valid pair $(a,b)$, we define the sequence $t_k$ by
\[
t_0 = b, \quad t_1 = 6a + 19b, \quad t_{k+2} = 38t_{k+1} - t_k.
\]

We compute $t_k \mod N$ and $t_k \mod p$ for
\[
N = 10^4, \quad p = 160001,
\]
and verify whether $t_k \equiv 0 \pmod{N}$ and $t_k \equiv \pm 10^m \pmod{p}$ can hold simultaneously as in the previous example.

This method successfully excludes most cases except for the following five cases:
\[
(a,b) = (22,8),\ (38,16),\ (68,24),\ (67,24),\ (3,12).
\]
We handle these remaining cases taking $p = 1601$ instead of $p=160001$, where the multiplicative subgroup $\{\pm 10^m \bmod p\}$ has size 200. For each of these five cases, we confirm that none of the $t_k$ values congruent with $0 \pmod{N}$ modulo $p$ belong to the subgroup $\{\pm 10^m \bmod p\}$, thus eliminating all possibilities. This completes the proof that no new values of the form $10^x - y^2$ appear in the interval $[0,2000]$ for $x \geq 9$. All computations were performed using Google Colab, and the corresponding notebook is available at:

\url{https://colab.research.google.com/drive/1c77O20BIBMl_UMuH4MmnhKy0POM9LBc-?usp=sharing}

\bibliographystyle{amsplain}

\begin{thebibliography}{99}

\bibitem{LeSoydan2020}
M.~Le and G.~Soydan,
\textit{A brief survey on the generalized Lebesgue--Ramanujan--Nagell equation},
Surv. Math. Appl. \textbf{15} (2020), 473--523.
ISSN 1842-6298 (electronic), 1843-7265 (print).

\bibitem{HercherFegert2025}
C.~Hercher and K.~Fegert,
\textit{Triangular numbers with a single repeated digit},
J. Integer Seq. \textbf{28} (2025), Article 25.2.1.

\end{thebibliography}

\end{document}